\documentclass[10pt, oneside]{amsart}

\title{Global transfer systems of abelian compact Lie groups}
\author{Miguel Barrero}
\address{IMAPP, Radboud University Nijmegen, The Netherlands}
\email{m.barrero@math.ru.nl}

\usepackage[english]{babel}
\usepackage{mathtools}
\usepackage[T1]{fontenc}
\usepackage{amsmath, amsthm, amsfonts, mathrsfs, amssymb}
\usepackage{tikz-cd}
\usepackage{quiver}
\usepackage[shortlabels]{enumitem}

\usepackage[a4paper, textheight = 25cm, textwidth = 16cm, top = 2.5cm, centering]{geometry}

\usepackage[backend=biber, style=alphabetic, citestyle=alphabetic, maxnames=5,minnames=3, maxalphanames=5, doi=false, isbn=false, url=false, eprint=false]{biblatex}
\usepackage{chngcntr}
\counterwithout{equation}{section}
\usepackage{hyperref}

\usepackage{pict2e}
\makeatletter
\newcommand{\adjunction}[4]{%
  #1\colon #2%
  \mathrel{\vcenter{%
    \offinterlineskip\m@th
    \ialign{%
      \hfil$##$\hfil\cr
      \longrightharpoonup\cr
      \noalign{\kern-.3ex}
      \smallbot\cr
      \longleftharpoondown\cr
    }%
  }}%
  #3 \noloc #4%
}
\newcommand{\longrightharpoonup}{\relbar\joinrel\rightharpoonup}
\newcommand{\longleftharpoondown}{\leftharpoondown\joinrel\relbar}
\newcommand\noloc{%
  \nobreak
  \mspace{6mu plus 1mu}
  {:}
  \nonscript\mkern-\thinmuskip
  \mathpunct{}
  \mspace{2mu}
}
\newcommand{\smallbot}{%
  \begingroup\setlength\unitlength{.15em}%
  \begin{picture}(1,1)
  \roundcap
  \polyline(0,0)(1,0)
  \polyline(0.5,0)(0.5,1)
  \end{picture}%
  \endgroup
}
\makeatother

\usetikzlibrary{decorations.markings}
\tikzset{double line with arrow/.style args={#1,#2}{decorate,decoration={markings,%
mark=at position 0 with {\coordinate (ta-base-1) at (0,1pt);
\coordinate (ta-base-2) at (0,-1pt);},
mark=at position 1 with {\draw[#1] (ta-base-1) -- (0,1pt);
\draw[#2] (ta-base-2) -- (0,-1pt);
}}}}
\tikzset{Equal/.style={-,double line with arrow={-,-}}}

\newtheorem{thm}{Theorem}[section]

\newtheorem{lemm}[thm]{Lemma}
\newtheorem{prop}[thm]{Proposition}
\newtheorem*{thmintro}{Theorem}
\theoremstyle{definition}
\newtheorem{defi}[thm]{Definition}

\theoremstyle{remark}
\newtheorem{constr}[thm]{Construction}
\theoremstyle{remark}

\theoremstyle{remark}
\newtheorem{rem}[thm]{Remark}
\theoremstyle{remark}

\DeclareMathAlphabet{\mathpzc}{OT1}{pzc}{m}{it}

\newcommand{\id}{\mathrm{id}}

\newcommand{\NN}{\mathbb{N}}
\newcommand{\ZZ}{\mathbb{Z}}

\newcommand{\opp}{\mathrm{op}}

\newcommand{\Nin}{N_\infty}

\newcommand{\IGc}{\overline{I}_G}

\newcommand{\All}{\mathpzc{All}}

\newcommand{\fami}{\mathpzc{F}}
\newcommand{\eami}{\mathpzc{E}}
\newcommand{\dami}{\mathpzc{D}}
\newcommand{\Ab}{\mathpzc{Ab}}
\newcommand{\Lie}{\mathpzc{Lie}}
\newcommand{\Abp}{{\Ab_p}}
\newcommand{\Abq}{{\Ab_q}}
\newcommand{\Cyc}{\mathpzc{Cyc}}
\newcommand{\Cycp}{{\Cyc_p}}
\newcommand{\AbLie}{\mathpzc{AbLie}}
\DeclareMathOperator{\Subg}{Sub}
\newcommand{\Gfam}{\mathpzc{G}}

\begin{document}
\begin{abstract}
Global transfer systems are equivalent to global $N_\infty$-operads, which parametrize different levels of commutativity in globally equivariant homotopy theory, where objects have compatible actions by all compact Lie groups. In this paper we explicitly describe and completely classify global transfer systems for the family of all abelian compact Lie groups.
\end{abstract}
\maketitle
\tableofcontents

\section{Introduction}

In this paper we introduce global transfer systems for a family of compact Lie groups. A global transfer system is an algebraic construction, a coherent choice of pairs of groups. Explicitly, a global transfer system for a family $\fami$ of compact Lie groups is a collection of pairs $(K, H)$ with $H$ belonging to the chosen family $\fami$, and $K$ a finite index closed subgroup of $H$. This collection is required to form a partial order on $\fami$, and to be closed under pullback along arbitrary continuous homomorphisms of compact Lie groups.

This notion of global transfer systems might be of independent interest, and to our knowledge, has not appeared in the literature before. This paper uses algebraic techniques to study these global transfer systems, as we consider the structure of the involved groups, and homomorphisms between them. The motivation for our interest in global transfer systems comes from equivariant homotopy theory, as we explain in more detail below. However, the reader is not expected to be familiar with equivariant homotopy theory.

Recently, in \cite{globalNin}, we defined global transfer systems, and we proved that the partially ordered set of global transfer systems for the family of all compact Lie groups is equivalent to the homotopy category of \emph{global $\Nin$-operads}. These different global $\Nin$-operads represent different levels of commutativity in global homotopy theory.

\subsection*{Global homotopy theory}

Global homotopy theory studies spaces and spectra which have compatible actions by all compact Lie groups. Each global space or global spectrum $X$ has an underlying $H$-space or $H$-spectrum $X_H$ for each compact Lie group $H$. The globally equivariant structure on $X$ also consists of compatible $K$-equivariant maps from $\alpha^\ast(X_H)$ to $X_K$ for each continuous homomorphism $\alpha \colon K \to H$ of compact Lie groups. Furthermore these $K$-equivariant maps are required to be equivalences whenever $\alpha$ is injective. Here $\alpha^\ast(X_H)$ denotes $X_H$ with the $K$-action obtained by pulling back the $H$-action along $\alpha$.

Linskens, Nardin, and Pol formalized this description of globally equivariant spaces and spectra in \cite{lnpglobal}. They showed that the preexisting models of unstable and stable global homotopy theory (Henriques and Gepner \cite{orbispaces}, and Schwede \cite{global}) are equivalent to a partially lax limit of a certain functor that sends a compact Lie group $G$ to the $\infty$-category of $G$-spaces, or the $\infty$-category of $G$-spectra in the stable case.

Each $G$-space or $G$-spectrum $X$ can also be interpreted as a collection of $H$-spaces or $H$-spectra $X_H$ with $H$ varying through the subgroups of $G$, by restricting the $G$-action to the subgroups. The conditions that these $X_H$ associated to a $G$-space or $G$-spectrum $X$ have to satisfy are weaker than the conditions that globally equivariant objects satisfy. This in particular means that not all $G$-spaces or $G$-spectra can appear as underlying a global space or global spectrum.

Many traditional $G$-equivariant constructions for varying $G$ assemble into a single globally equivariant object, for example real and complex equivariant bordism and equivariant topological $K$-theory. By using this additional structure encoded by the global compatibility, one can obtain strong results in equivariant homotopy theory, such as Hausmann's proof \cite{hausmann2022global} of the equivariant analog of Quillen's theorem on formal group laws for abelian compact Lie groups.

\subsection*{$G$-transfer systems}

The analogs of global transfer systems with respect to a single group $G$ have been previously studied in $G$-equivariant homotopy theory. These are called \emph{$G$-transfer systems}, and they are less restrictive than global transfer systems. This is due to the stronger compatibility between the various $X_H$ underlying a global space or global spectrum, when compared to a $G$-space or $G$-spectrum. Concretely, global transfer systems are required to be closed under pullback along arbitrary homomorphisms, while $G$-transfer systems are only required to be closed under restriction to subgroups and conjugation by elements of $G$.

Considering objects with a multiplication, non-equivariantly we say that a multiplication is $E_\infty$ if it is associative, commutative and unital up to all higher homotopies. In $G$-equivariant homotopy theory there are different non-equivalent analogs of $E_\infty$-multiplications. These are represented by different $\Nin$-operads, introduced by Blumberg and Hill in \cite{BLUMBERG2015658}, and they model different levels of commutativity in the equivariant world. Each $\Nin$-operad has an associated $G$-transfer system. $G$-spectra which are algebras over $\Nin$-operads are all $E_\infty$-ring spectra when one forgets the group actions, but they posses additional structure, encoded by the existence or non-existence of those \emph{norm maps} between their equivariant homotopy groups prescribed by the associated $G$-transfer systems. These assemble into \emph{incomplete Tambara functors} \cite[Theorem~4.14]{tambara}. In the unstable setting this additional structure on algebras is encoded by \emph{transfer maps} which assemble into \emph{incomplete Mackey functors} \cite[Proposition~1.4]{mackey}.

Blumberg and Hill first conjectured that these $\Nin$-operads are completely classified by their associated $G$-transfer systems. The last step of this conjecture was proven separately by Gutiérrez and White \cite{GutierrezWhite}, by Rubin \cite{rubin1}, and by Bonventre and Pereira \cite{BonventrePereira}, who constructed examples of all possible $\Nin$-operads using different methods.

Various authors have studied $G$-transfer systems, and explicitly computed the partially ordered set of $G$-transfer systems for specific groups $G$, see \cite{balchin2021ninftyoperads}, \cite{balchinCpqr}, \cite{Rubin2}, \cite{transfersystemspmqn}, or \cite{transfersduality}. An explicit description of the partially ordered set of $G$-transfer systems is currently unknown for most groups $G$. With this paper, we introduce the analogous study of global transfer systems, which is much more accessible than the study of $G$-transfer systems due to the stronger conditions that global transfer systems have to satisfy.

\subsection*{Content of this paper}

In this paper we introduce global transfer systems for smaller families than the family of all compact Lie groups. For each family $\fami$ of compact Lie groups, the global transfer systems for $\fami$ form a partially ordered set. We explicitly compute the partially ordered set of global transfer systems of abelian compact Lie groups, denoted by $I_\AbLie$, and give an explicit description of these global transfer systems.

To accomplish this, we study the relation between global transfer systems for different families of compact Lie groups. A global transfer system for a family can always be extended to any bigger family. We prove that a global transfer system of abelian compact Lie groups is completely determined by its restriction to finite cyclic groups, and so in this case this extension is unique. This shows that $I_\AbLie$ is isomorphic to $I_\Cyc$ the partially ordered set of global transfer systems of finite cyclic groups, and we prove that the latter splits as a product indexed on the set of prime numbers. 

\begin{thmintro}[Theorem~\ref{thmmain}]
The partially ordered set $I_\AbLie$ of global transfer systems for the family of abelian compact Lie groups is isomorphic to the partially ordered set $I_\Cyc$ of global transfer systems for the family of finite cyclic groups, which is itself isomorphic to the product $(\prod_{p \in P} \NN \cup \{\infty\})^\opp$, where $P$ is the set of prime numbers.
\end{thmintro}

We prove this theorem in four steps. First in Section~\ref{sectionCycp} we compute $I_\Cycp$ the partially ordered set of transfer systems of finite cyclic $p$-groups for a prime $p$. Then in Section~\ref{sectionAbp} we prove that this is isomorphic to $I_\Abp$ the partially ordered set of transfer systems of finite abelian $p$-groups. Thirdly we check in Section~\ref{sectionAb} that the partially ordered set $I_\Ab$ of transfer systems of finite abelian groups is isomorphic to $\prod_{p \in P} I_\Abp$. Finally, with a simple argument we show in Section~\ref{sectionAbLie} that $I_\AbLie \cong I_\Ab$.

We also give an explicit description of the pairs contained in each global transfer system of abelian compact Lie groups. We accomplish this by interpreting an element of $(\prod_{p \in P} \NN \cup \{\infty\})^\opp$ as a "generalized prime factorization", and relating this to the orders of the elements of the involved groups. Our explicit description of global transfer systems could be helpful in constructing globally equivariant objects that exhibit intermediate levels of commutativity.

As we showed in \cite[Theorem~6.7]{globalNin}, the partially ordered set $I_\Lie$ of global transfer systems for the family of all compact Lie groups parametrizes the different levels of globally equivariant commutativity. Our result is a substantial step towards obtaining an explicit description of the partially ordered set $I_\Lie$. However, as we show in Section~\ref{sectionnonabelian}, $I_\Lie$ is strictly larger than $I_\AbLie$. More concretely, there are different global transfer systems of compact Lie groups which coincide when restricted to abelian compact Lie groups.

\subsection*{Conventions}

In this paper we only ever consider closed subgroups of compact Lie groups and continuous homomorphisms between them, which are therefore smooth.

\subsection*{Acknowledgements}

The work presented in this paper is part of my PhD project. I would like to thank my PhD supervisor Magdalena Kędziorek for her advice on both the mathematical content and presentation of this paper. I would also like to thank Maarten Solleveld and Tommy Lundemo for several helpful conversations, and the anonymous referee for their helpful suggestions.

\section{Global transfer systems for a family \texorpdfstring{$\fami$}{F}}

We begin by defining global transfer systems for a family of compact Lie groups, generalizing Definition~4.2 of \cite{globalNin} to a smaller family than that of all compact Lie groups.

\begin{defi}
A \emph{family of compact Lie groups} $\fami$ (from here on just a \emph{family}) is a non-empty class of compact Lie groups closed under isomorphisms and closed subgroups.
\end{defi}

\begin{rem}
A \emph{global family} in the sense of \cite[Definition~1.4.1]{global} is additionally required to be closed under quotient subgroups. One can do global homotopy theory with respect to any global family of compact Lie groups (see \cite[Subsections~1.4~and~4.3]{global}), but for the contents of this paper we can consider more general families.
\end{rem}

Some relevant families that we will consider are $\Lie$ the family of all compact Lie groups, $\AbLie$ the family of abelian compact Lie groups, $\Ab$ the family of finite abelian groups, $\Cyc$ the family of finite cyclic groups, $\Abp$ the family of finite abelian $p$-groups for a prime $p$, and $\Cycp$ the family of finite cyclic $p$-groups. A finite group is a $0$-dimensional compact Lie group when given the discrete topology.

Given a compact Lie group $G$, we will use $\Gfam$ to denote the family of groups isomorphic to a closed subgroup of $G$.

\begin{defi}
\label{defitransfer}
Let $\fami$ be a family of compact Lie groups. A \emph{global $\fami$-transfer system} $T$ (or more concisely an \emph{$\fami$-transfer system}) is a binary relation $\leqslant_T$ on $\fami$ such that:
\begin{enumerate}[\roman*)]
    \item The relation $\leqslant_T$ is a partial order that refines the finite index closed subgroup relation on $\fami$. This means that it is transitive, reflexive and antisymmetric, and that if $ K \leqslant_T H$ then $K$ is a finite index closed subgroup of $H$.
    \item \label{defitransferii} The relation $\leqslant_T$ is closed under pullback along arbitrary continuous homomorphisms. Explicitly, for any $G,H, K \in \fami$ and any continuous homomorphism $\theta \colon G \to H$, if $K \leqslant_T H$ then $\theta^{-1}(K) \leqslant_T G$. Note that $\lvert G/\theta^{-1}(K) \rvert \leqslant \lvert H/K \rvert$, thus $\theta^{-1}(K)$ has finite index in $G$.
\end{enumerate}
\end{defi}

When referring to global transfer systems for a specific family, like $\Ab$ the family of finite abelian groups, we will often write "global transfer systems of finite abelian groups" instead of "$\Ab$-transfer systems" for clarity.

\begin{lemm}
The $\fami$-transfer systems form a poset with respect to inclusion, which we denote by $I_\fami$. The poset $I_\fami$ is in fact a complete lattice.
\end{lemm}
\begin{proof}
Explicitly, we have that $T \subset T'$ for $\fami$-transfer systems $T$ and $T'$ if $K \leqslant_T H$ implies that $K \leqslant_{T'} H$. Condition ii) implies that continuous isomorphisms of pairs of compact Lie groups preserve the relation $\leqslant_T$, and there are only countably many isomorphism classes of compact Lie groups, thus $I_\fami$ is indeed a set. The inclusion relation gives a partial order on the set $I_\fami$.

Arbitrary intersections of $\fami$-transfer systems are again $\fami$-transfer systems, and these give the meets of the poset $I_\fami$. There is a greatest $\fami$-transfer system $\All$, for which $K \leqslant_\All H$ for any finite index closed subgroup $K \leqslant H \in \fami$. Therefore $I_\fami$ also has arbitrary joins, and is thus a complete lattice. Joins in $I_\fami$ are the transitive closures of unions of $\fami$-transfer systems.
\end{proof}

A poset can be viewed as a small category with at most one morphism between any two objects. Under this identification a functor between posets is the same thing as an order preserving function, and an equivalence of categories is an isomorphism of posets. We will use these identifications implicitly throughout this article.

The definition of an $\fami$-transfer system immediately leads to the following properties.

\begin{lemm}
\label{lemmsection}
Let $\fami$ be a family, and let $T$ be an $\fami$-transfer system.
\begin{enumerate}[label = \roman*)]
    \item Let $\pi \colon G \to H$ be a continuous homomorphism between groups in $\fami$, and let $s \colon H \to G$ be a section of $\pi$, that is, a homomorphism such that $\pi \circ s = \id_H$. Then $K \leqslant_T H$ if and only if $\pi^{-1}(K) \leqslant_T G$.
    \item Let $H, H'$ be groups in $\fami$, and assume that $K \leqslant_T H$ and $K' \leqslant_T H'$. Then $K\times K' \leqslant_T H \times H'$.
    \item Let $\prod_{i \geqslant 0}^N H_i$ be a finite product of groups in $\fami$, and let $K_i \leqslant H_i$ for each $0 \leqslant i \leqslant N$. If $K_i \leqslant_T H_i$ for each $0 \leqslant i \leqslant N$, then $(\prod_{i \geqslant 0}^N K_i) \leqslant_T (\prod_{i \geqslant 0}^N H_i)$.
\end{enumerate}
\end{lemm}
\begin{proof}
\begin{enumerate}[label = \roman*)]
    \item If $K \leqslant_T H$ then $\pi^{-1}(K) \leqslant_T G$ by Condition~\ref{defitransferii} of the Definition~\ref{defitransfer} of $\fami$-transfer system. Similarly if $\pi^{-1}(K) \leqslant_T G$ then $s^{-1}(\pi^{-1}(K)) = K \leqslant_T H$.
    \item Since $K \leqslant_T H$, by pulling back along the projection $\pi_H \colon H \times H' \to H$ we get that $(K \times H') \leqslant_T (H \times H')$. Similarly since $K' \leqslant_T H'$ then $(K \times K') \leqslant_T (K \times H')$, and transitivity of $T$ proves that $(K \times K') \leqslant_T (H \times H')$.
    \item By induction and part ii) we obtain the statement. \qedhere
\end{enumerate}
\end{proof}

Blumberg and Hill introduced $G$-equivariant $\Nin$-operads in \cite{BLUMBERG2015658}. These operads classify intermediate levels of commutativity in equivariant homotopy theory. They originally conjectured in \cite{BLUMBERG2015658} that these $\Nin$-operads are equivalent to \emph{$G$-indexing systems}. The last step of this equivalence was later proven separately by  Gutiérrez and White in \cite{GutierrezWhite}, by Rubin in \cite{rubin1}, and by Bonventre and Pereira in \cite{BonventrePereira}. It was pointed out in \cite[Section~3]{Rubin2} and \cite[Lemma 6]{balchin2021ninftyoperads} that these $G$-indexing systems are equivalent to \emph{$G$-transfer systems}, which have the following simpler definition. 

\begin{defi}
\label{defiGtransfer}
Let $G$ be a compact Lie group. A $G$-transfer system $T$ is a binary relation on $\Subg(G)$ the set of closed subgroups of $G$, such that:
\begin{enumerate}[\roman*)]
    \item The relation $\leqslant_T$ is a partial order that refines the finite index closed subgroup relation on $\Subg(G)$. This means that it is transitive, reflexive and antisymmetric, and that if $ K \leqslant_T H$ then $K$ is a finite index closed subgroup of $H$.
    \item The relation $\leqslant_T$ is closed under restriction to closed subgroups. Explicitly, for any $H, K, L \in \Subg(G)$ with $L \leqslant H$, if $K \leqslant_T H$ then $(K \cap L) \leqslant_T L$. 
    \item The relation $\leqslant_T$ is closed under conjugation. Explicitly, for any $H, K \in \Subg(G)$ and $g \in G$, if $K \leqslant_T H$ then $(g K g^{-1}) \leqslant_T (g H g^{-1})$.
\end{enumerate}
\end{defi}

Our definition of global $\fami$-transfer systems in Definition~\ref{defitransfer} is the globally equivariant analog of this definition of $G$-transfer systems, which is the generalization of \cite[Definition~3.4]{Rubin2} to compact Lie groups.

\begin{rem}
\label{remdefGtransfer}
We want to make it clear that $G$-transfer systems are not the same as $\Gfam$-transfer systems, the global $\Gfam$-transfer systems associated to the family $\Gfam$ of groups isomorphic to a closed subgroup of $G$. Definition~\ref{defiGtransfer} only requires closure under restriction to a subgroup and conjugation by elements of $G$, while Condition~\ref{defitransferii} of Definition~\ref{defitransfer} requires closure under arbitrary continuous homomorphisms between subgroups of $G$. This means that any $\Gfam$-transfer system is a $G$-transfer system, but the opposite is not true. See Remark~\ref{remUgRg} for more details about this difference between $\Gfam$-transfer systems and $G$-transfer systems.

As an example, let $p$ be a prime, and consider the cyclic group $C_{p^2}$. Let $T$ be the partial order on $\Subg(C_{p^2})$ such that $e \leqslant_T C_p$ is the only non-reflexive relation. This is a $C_{p^2}$-transfer system, as mentioned in \cite[Example~14]{balchin2021ninftyoperads}, because it satisfies Definition~\ref{defiGtransfer}. However it is not closed under the quotient homomorphism $C_{p^2} \to C_p$, and therefore it does not satisfy Condition~\ref{defitransferii} of Definition~\ref{defitransfer}.
\end{rem}

\section{Global transfer systems for \texorpdfstring{$\Cycp$}{Cycp}}
\label{sectionCycp}

Throughout this section we fix a prime $p$. We begin by computing the poset of global transfer systems of finite cyclic $p$-groups, denoted by $I_\Cycp$. The finite cyclic $p$-groups are all isomorphic to $C_{p^n}$ for some $n \geqslant 0$, and $C_{p^n}$ has exactly one subgroup isomorphic to $C_{p^m}$ for each $0 \leqslant m \leqslant n$. Balchin, Barnes and Roitzheim in \cite{balchin2021ninftyoperads} computed the poset of $C_{p^n}$-transfer systems. Although we do not explicitly use their results here, the content of this section is based on their work.

As mentioned already in Remark~\ref{remdefGtransfer}, there are more $C_{p^n}$-transfer systems than global transfer systems for the family of subgroups of $C_{p^n}$, so computing the poset $I_\Cycp$ of $\Cycp$-transfer systems is much simpler than computing the possible $C_{p^n}$-transfer systems for a given $C_{p^n}$.

\begin{lemm}
\label{lemmcyc}
Let $T$ be a $\Cycp$-transfer system, and consider $0 \leqslant m_0 < n_0$ such that $C_{p^{m_0}} \leqslant_T C_{p^{n_0}}$. Then $C_{p^{m}} \leqslant_T C_{p^{n}}$ for any $m \geqslant m_0$ and any $n \geqslant m$.
\end{lemm}
\begin{proof}
Let $\theta$ be the subgroup inclusion $\theta \colon C_{p^{m_0 + 1}} \to C_{p^{n_0}}$. Pullback along $\theta$ implies that $C_{p^{m_0}} \leqslant_T C_{p^{m_0 + 1}}$. For each $i \geqslant 0$ pullback along the projection $C_{p^{m_0 + i + 1}} \to C_{p^{m_0 + 1}}$ with kernel $C_{p^i}$ shows that $C_{p^{m_0 + i}} \leqslant_T C_{p^{m_0 + i + 1}}$. Finally by transitivity of $T$, $C_{p^{m}} \leqslant_T C_{p^{n}}$ for any $m \geqslant m_0$ and any $n \geqslant m$.
\end{proof}

\begin{constr}
\label{constrIn}
For each $n \in \NN$, let $T^n$ be the $\Cycp$-transfer system such that $C_{p^{i}} \leqslant_{T^n} C_{p^{j}}$ precisely if $i = j$ or $n \leqslant i \leqslant j$. In particular, $T^0= \All$ the greatest $\Cycp$-transfer system. It is straightforward to check that for each $n \in \NN$ this definition gives a partial order that refines the subgroup relation. For any homomorphism $\theta \colon C_{p^{m}} \to C_{p^{j}}$, if $n \leqslant i \leqslant j$ then $\theta(C_{p^{n}}) \leqslant C_{p^{i}}$ as subgroups of $C_{p^{j}}$. This means that $C_{p^{n}} \leqslant \theta^{-1}(C_{p^{i}})$, so that $\theta^{-1}(C_{p^{i}}) \leqslant_{T^n} C_{p^{m}}$ by construction, and so $T^n$ satisfies Condition ii) of Definition~\ref{defitransfer}. Note that if $n\leqslant n'$ then $T^{n'} \subset T^n$. 

Let $T^\infty$ denote the smallest $\Cycp$-transfer system, such that if $K \leqslant_{T^\infty} H$ then $K=H$. Note that $T^\infty = \cap_{n \in \NN} T^n$, so $T^\infty$ is the meet of the set $\{T^n\}_{n \in \NN}$ in $I_\Cycp$.
\end{constr}

\begin{prop}
\label{propcycp}
The poset $I_\Cycp$ of $\Cycp$-transfer systems is isomorphic to $(\NN \cup \{\infty\})^\opp$ via the functor $(\NN \cup \{\infty\})^\opp \to I_\Cycp$ that sends $n$ to $T^n$.
\end{prop}
\begin{proof}
By construction, the poset $\{T^n \mid n \in \NN \cup \{\infty\} \}$ ordered by inclusion is isomorphic to $(\NN \cup \{\infty\})^\opp$. Therefore we only need to prove that any $\Cycp$-transfer system equals one of these $T^n$.

Let $T$ be a $\Cycp$-transfer system. If $T \neq T^\infty$, let $n \in \NN$ be minimal such that $C_{p^{n}} \leqslant_T C_{p^{m}}$ for some $m > n$. This means that $C_{p^{i}} \nleqslant_T C_{p^{j}}$ for each $i < n$ and each $j>i$. On the other hand $C_{p^{i}} \leqslant_T C_{p^{j}}$ for each $j \geqslant i \geqslant n$ by Lemma~\ref{lemmcyc}. Therefore $T = T^n$.
\end{proof}

\section{Change of family functors}
\label{sectionchange}

We now construct and study change of family functors for global transfer systems for different families of compact Lie groups. We will prove in particular that if the family $\eami$ is contained in the family $\fami$ then the poset $I_\eami$ of $\eami$-transfer systems embeds into the poset $I_\fami$ of $\fami$-transfer systems.

\begin{defi}
\label{defiR}
Let $\eami, \fami$ be families of compact Lie groups. We define a functor $R_\eami^\fami \colon I_\eami \to I_\fami$. It sends an $\eami$-transfer system $T$ to the $\fami$-transfer system $S= R_\eami^\fami(T)$ that satisfies that $K \leqslant_S H$ if and only if:
\begin{enumerate}[\roman*)]
    \item $K, H \in \fami$ and $K$ is a finite index closed subgroup of $H$.
    \item For each $\theta \colon  G \to H$ with $G \in \eami$, $\theta^{-1}(K) \leqslant_T G$.
\end{enumerate}
It is straightforward to check that this indeed gives an $\fami$-transfer system, and that $R_\eami^\fami$ preserves inclusion of global transfer systems and thus is a functor between the posets $I_\eami$ and $I_\fami$. Note that by construction the functor $R_\fami^\fami$ is always the identity on $I_\fami$.
\end{defi}

\begin{lemm}
\label{lemmU}
If the family $\eami$ is contained in the family $\fami$ then the functor $R_\fami^\eami$ is precisely the forgetful functor from $I_\fami$ to $I_\eami$, which we will denote with $U_\eami^\fami \colon I_\fami \to I_\eami$, and which sends an $\fami$-transfer system $T$ to the restriction of the relation $T$ from $\fami$ to $\eami$.
\end{lemm}
\begin{proof}
Let $T$ be an $\fami$-transfer system, and let $K \leqslant H$ be groups in $\eami$. If $K \leqslant_T H$ then $K \leqslant_{R_\fami^\eami(T)} H$ because $T$ is closed under pullback along arbitrary homomorphisms. Conversely if $K \leqslant_{R_\fami^\eami(T)} H$ then taking $\theta=\id_H$ in the construction of $R_\fami^\eami$ shows that $K \leqslant_T H$.
\end{proof}

\begin{lemm}
\label{lemmRcomp}
Let $\dami, \eami, \fami$ be families of compact Lie groups. If any homomorphism from a group in $\dami$ to a group in $\fami$ factors through a group belonging to $\eami$, then the composition $R_\eami^\fami \circ R_\dami^\eami$ equals $R_\dami^\fami$.
\end{lemm}
\begin{proof}
For $T$ a $\dami$-transfer system, by Definition~\ref{defiR} we see that 
\begin{equation}
\label{eqdirect}
    K \leqslant_{R^\fami_\dami(T)} H \; \text{if and only if for each} \; \theta \colon  G \to H \; \text{with} \; G \in \dami , \; \theta^{-1}(K) \leqslant_T G.
\end{equation}
Meanwhile using twice the construction given in Definition~\ref{defiR} we see that
\begin{equation}
\label{eqcomp}
\begin{split}
    K \leqslant_{R^\fami_\eami(R^\eami_\dami(T))} H \; \text{if and only if for each} \; \theta_1 \colon  L \to H \; \text{with} \; L \in \eami \\
    \text{and each} \; \theta_2 \colon  G \to L \; \text{with} \; G \in \dami, \; \text{we have that} \; \theta_2^{-1}(\theta_1^{-1}(K)) \leqslant_T G.
\end{split}
\end{equation}

It is straightforward that if $K$ and $H$ satisfy (\ref{eqdirect}) then they satisfy (\ref{eqcomp}) since $\theta_2^{-1}(\theta_1^{-1}(K))= (\theta_1 \circ \theta_2)^{-1}(K)$. Conversely assume that $K$ and $H$ satisfy (\ref{eqcomp}), and let $\theta \colon  G \to H$ be a homomorphism with $G \in \dami$. By the condition in the statement of the lemma, $\theta$ factors through a group belonging to $\eami$ as $\theta = \theta_1 \circ \theta_2$. Therefore by (\ref{eqcomp}) we know that $\theta^{-1}(K) = \theta_2^{-1}(\theta_1^{-1}(K)) \leqslant_T G$.
\end{proof}

\begin{prop}
\label{propadjoint}
If the family $\eami$ is contained in the family $\fami$ then $U_\eami^\fami \circ R_\eami^\fami$ is the identity on $I_\eami$, and $R_\eami^\fami$ is right adjoint to $U_\eami^\fami$. In particular, $R_\eami^\fami$ is an embedding of $I_\eami$ into $I_\fami$.
\end{prop}
\begin{proof}
Any homomorphism between groups in $\eami$ factors through a group belonging to $\fami$ as $\eami \subset \fami$. Thus by Lemma~\ref{lemmRcomp} $U_\eami^\fami \circ R_\eami^\fami = R_\fami^\eami \circ R_\eami^\fami = R_\eami^\eami$ which is the identity on $I_\eami$.

Let $T \in I_\fami$ and $T' \in I_\eami$. If $T \subset R_\eami^\fami(T')$, applying $U_\eami^\fami$ shows that $U_\eami^\fami(T) \subset T'$. Conversely if $U_\eami^\fami(T) \subset T'$, let $K \leqslant_T H$. For each $\theta \colon  G \to H$ with $G \in \eami$, we have that $\theta^{-1}(K) \leqslant_T G$ because $T$ is an $\fami$-transfer system, so by definition $\theta^{-1}(K) \leqslant_{U_\eami^\fami(T)} G$ and thus $\theta^{-1}(K) \leqslant_{T'} G$ because $U_\eami^\fami(T) \subset T'$. This implies that $K \leqslant_{R_\eami^\fami(T')} H$, completing the proof that $U_\eami^\fami(T) \subset T'$ if and only if $T \subset R_\eami^\fami(T')$.
\end{proof}

\begin{rem}
The adjointness of the functors $R_\eami^\fami$ and $U_\eami^\fami$ for families $\eami \subset \fami$ and the completeness of the lattice $I_\fami$ means that for an $\eami$-transfer system $T$, the $\fami$-transfer system $R_\eami^\fami(T)$ is precisely the maximum of the $\fami$-transfer systems $S$ such that $U_\eami^\fami(S) \subset T$.
\end{rem}

\begin{rem}
\label{remUgRg}
In \cite[Proposition~6.6]{globalNin} we constructed a functor $R_G$ from the poset $I_G$ of $G$-transfer systems to the poset $I_\Lie$ of global transfer systems for the family of all compact Lie groups. This functor is right adjoint to the forgetful functor from $I_\Lie$ to $I_G$. The image (denoted with $\IGc$ in \cite[Definition~7.1]{globalNin}) of this forgetful functor consists precisely of the $\Gfam$-transfer systems, and the restriction of the functor $R_G$ to $\IGc = I_\Gfam$ is exactly the functor $R^\Lie_\Gfam$ that we have constructed in this section. What this means is that for any compact Lie group $G$, the $\Gfam$-transfer systems are precisely the $G$-transfer systems that can appear underlying a global $\Lie$-transfer system. 

A $G$-space, when regarded as a functor from the orbit category of $G$ to spaces, includes as part of its structure maps associated to restriction to subgroups and conjugation by elements of $G$, as these generate the morphisms of the orbit category of $G$. A globally equivariant object includes as part of its structure a map associated to any continuous homomorphism $f \colon K \to H$, as we mentioned in the introduction. As explained in Remark~\ref{remdefGtransfer}, the difference between $G$-transfer systems and global $\Gfam$-transfer systems is that the former only have to be closed under restriction to subgroups and conjugation by elements of $G$, while the latter additionally have to be closed under arbitrary continuous homomorphisms. This corresponds directly to the additional structure that globally equivariant objects posses in comparison to $G$-equivariant objects.
\end{rem}

\section{Global transfer systems for \texorpdfstring{$\Abp$}{Abp}}
\label{sectionAbp}

In this section we again fix a prime $p$. Let $\Abp$ denote the family of finite abelian $p$-groups.

\begin{prop}
\label{propRTn}
Let $T^n$ for $n \in \NN \cup \{\infty\}$ be the $\Cycp$-transfer systems defined in Construction~\ref{constrIn}, and for each $n \in \NN \cup \{\infty\}$ let $A^n$ be the $\Abp$-transfer system $R_\Cycp^\Abp(T^n)$. Then $K \leqslant_{A^n} H$ if and only if $K$ contains all elements of $H$ of order less than or equal to $p^n$ (if $n = \infty$ let $p^\infty = \infty$).
\end{prop}
\begin{proof}
By the construction of the functor $R_\Cycp^\Abp$, for finite abelian $p$-groups $K \leqslant H$ we know that $K \nleqslant_{A^n} H$ if and only if there exists a homomorphism $\theta \colon C_{p^m} \to H$ such that $\theta^{-1}(K) \nleqslant_{T^n} C_{p^m}$. By the definition of the transfer systems $T^n$ in Construction~\ref{constrIn}, $\theta^{-1}(K) \nleqslant_{T^n} C_{p^m}$ if and only if $\theta^{-1}(K)= C_{p^i}$ with $i < n$, and this is the case if and only if there exists some $x \in C_{p^m} \backslash \theta^{-1}(K)$ with order less than or equal to $p^n$. 

If there exists such a homomorphism $\theta \colon C_{p^m} \to H$ and $x \in C_{p^m} \backslash \theta^{-1}(K)$ with $\lvert x\rvert \leqslant p^n$, then $\theta(x) \in H \backslash K$ and has order less than or equal to $p^n$, which shows one implication. On the other hand, if there exists an element $y \in H \backslash K$ of order less than or equal to $p^n$, it determines a homomorphism $\theta \colon C_{\lvert y \rvert} \to H$ with $\theta(1) = y$, and then we have that $1 \in C_{\lvert y \rvert} \backslash \theta^{-1}(K)$ and $\lvert 1 \rvert = \lvert y \rvert \leqslant p^n$. By the previous paragraph this means that $K \nleqslant_{A^n} H$.
\end{proof}

\begin{prop}
\label{propsuborderless}
Let $T$ be an $\Abp$-transfer system with $U_\Cycp^\Abp(T)=T^n$. For any $H \in \Abp$ let $S(n, H) \leqslant H$ denote the subgroup composed of the elements of $H$ of order less than or equal to $p^n$. Then $S(n, H) \leqslant_T H$.
\end{prop}
\begin{proof}
The group $H$ is of the form $\prod_{j \geqslant 0}^N C_{p^{m_j}}$ for $N, m_1, \dots, m_N \in \NN$, and $S(n, H)$ is precisely $\prod_{j \geqslant 0}^N C_{p^{\min(m_j, n)}} \leqslant \prod_{j \geqslant 0}^N C_{p^{m_j}}$. 

For every $0 \leqslant j \leqslant N$, we know that if $m_j > n$ then $C_{p^{\min(m_j, n)}} \leqslant_{T^n} C_{p^{m_j}}$ by construction of $T^n$, and if $m_j \leqslant n$ then $C_{p^{\min(m_j, n)}} = C_{p^{m_j}} \leqslant_{T^n} C_{p^{m_j}}$. By Lemma~\ref{lemmsection} iii) we obtain that $S(n, H) \leqslant_T H$.
\end{proof}

\begin{prop}
\label{propabp}
The functors $U_\Cycp^\Abp$ and $R_\Cycp^\Abp$ are inverse isomorphisms between the poset $I_\Cycp$ and the poset $I_\Abp$ of $\Abp$-transfer systems.
\end{prop}
\begin{proof}
By Proposition~\ref{propadjoint} the functor $U_\Cycp^\Abp \colon I_\Abp \to I_\Cycp$ is left adjoint to $R_\Cycp^\Abp \colon I_\Cycp \to I_\Abp$, and $U_\Cycp^\Abp \circ R_\Cycp^\Abp = \id_{I_\Cycp}$. Let $T \in I_\Abp$, by adjointness $T \subset R_\Cycp^\Abp(U_\Cycp^\Abp(T))$. We know by the classification of $\Cycp$-transfer systems given in Proposition~\ref{propcycp} that $U_\Cycp^\Abp(T)=T^n$ for some $n \in \NN \cup \{\infty\}$. Let $A^n = R_\Cycp^\Abp(T^n) = R_\Cycp^\Abp(U_\Cycp^\Abp(T))$. We only have left to show that $A^n \subset T$.

Assume that $K \leqslant_{A^n} H$. By Proposition~\ref{propRTn} this means that $K$ contains all elements of $H$ of order less than or equal to $p^n$. Let $p^n K \leqslant K \leqslant H$ be the subgroup obtained by multiplication with $p^n$ under the $\ZZ$-module structure of $H$ (if $n = \infty$ let $p^n K$ denote the trivial subgroup). Let $f \colon H \to H / (p^n K)$ denote the projection to the quotient. Finally let $S= S(n, H / (p^n K)) \leqslant H / (p^n K)$ denote the subgroup composed of the elements of order less than or equal to $p^n$ of the group $H / (p^n K)$. By Proposition~\ref{propsuborderless} we know that $S \leqslant_T H / (p^n K)$.

Now let $x \in K$, then $f(p^n x)= 0$ and so $f(x) \in S$. Conversely, for $y \in H$, if $f(y) \in S= S(n, H / (p^n K))$ then $p^n f(y) = 0$, so $p^n y \in p^n K$. This means that there exists some $k \in K$ such that $p^n y = p^n k$. In that case $y - k$ has order less than or equal to $p^n$, so it belongs to $K$ as it contains all elements of $H$ of order less than or equal to $p^n$ by Proposition~\ref{propRTn}, and thus $y \in K$. Therefore $f^{-1}(S) = K$, and by pulling back along $f$ since $S \leqslant_T H / (p^n K)$ we obtain that $K \leqslant_T H$. This shows that $A^n = R_\Cycp^\Abp(U_\Cycp^\Abp(T)) \subset T$, and we conclude that $R_\Cycp^\Abp(U_\Cycp^\Abp(T)) = T$ and thus $U_\Cycp^\Abp$ and $R_\Cycp^\Abp$ are inverse isomorphisms.
\end{proof}

\section{Global transfer systems for \texorpdfstring{$\Ab$}{Ab}}
\label{sectionAb}

Let $P$ denote the set of all prime numbers and let $\Ab$ denote the family of finite abelian groups. Let $U \colon I_\Ab \to \prod_{p \in P} I_\Abp$ denote the functor induced by the functors $U^\Ab_\Abp \colon I_\Ab \to I_\Abp$. Let $R \colon  \prod_{p \in P} I_\Abp \to I_\Ab$ denote the functor given by $R((T_p)_{p \in P})= \cap_{p \in P} R^\Ab_\Abp(T_p)$.

\begin{prop}
\label{propab}
The functors $U$ and $R$ are inverse isomorphisms between the poset $\prod_{p \in P} I_\Abp$ and the poset $I_\Ab$ of $\Ab$-transfer systems.
\end{prop}
\begin{proof}
We first check that $U$ is left adjoint to $R$. Let $T$ be an $\Ab$-transfer system, and let $(T_p)_{p \in P} \in \prod_{p \in P} I_\Abp$. Then $T \subset R((T_p)_{p \in P})$ if and only if $T \subset R^\Ab_\Abp(T_p)$ for each prime $p$, which by adjointness of $U^\Ab_\Abp$ and $R^\Ab_\Abp$ holds if and only if $U(T) \subset (T_p)_{p \in P}$.

Next, by adjointness we know that $T \subset R(U(T))$ for any $T \in I_\Ab$. Let $K \leqslant H$ be finite abelian groups such that $K \leqslant_{R(U(T))} H$. The finite abelian group $H$ decomposes as a product of finite abelian $p$-groups $\prod_{p \in P} H_p$. Let $K_p=K \cap H_p$, then we also have that $K=\prod_{p \in P} K_p$. Let $\theta \colon H_p \to H$ denote the inclusion. Then $\theta^{-1}(K)=K_p$, so $K_p \leqslant_{R(U(T))} H_p$, thus $K_p \leqslant_{R^\Ab_\Abp(U^\Ab_\Abp(T))} H_p$ and $K_p \leqslant_{U^\Ab_\Abp(T)} H_p$ because $K_p, H_p \in \Abp$ and $U_\Abp^\Ab \circ R_\Abp^\Ab = \id_\Abp$. This implies that $K_p \leqslant_T H_p$ for each prime $p$, and by Lemma~\ref{lemmsection} iii) we obtain that $K \leqslant_T H$, and thus $R(U(T)) = T$.

For the other composition, we note that again by adjointness $U(R((T_p)_{p \in P})) \subset (T_p)_{p \in P}$ for each $(T_p)_{p \in P} \in \prod_{p \in P} I_\Abp$. For each prime let $K_p\leqslant H_p$ be finite abelian $p$-groups such that $K_p \leqslant_{T_p} H_p$. Then $K_p \leqslant_{R^\Ab_\Abp(T_p)} H_p$ since $U_\Abp^\Ab \circ R_\Abp^\Ab = \id_\Abp$, and $K_p \leqslant_{R^\Ab_\Abq(T_q)} H_p$ for $q \neq p$ by construction of $R^\Ab_\Abq$, as there are no non-trivial homomorphisms from $q$-groups to $p$-groups. Thus $K_p \leqslant_{R((T_p)_{p \in P})} H_p$ so $T_p \subset U^\Ab_\Abp(R((T_p)_{p \in P}))$ and $(T_p)_{p \in P} = U(R((T_p)_{p \in P}))$.
\end{proof}

\begin{rem}
Note that this splitting of transfer systems with respect to different primes does not happen in the $G$-equivariant case. It is only caused by the additional requirements of global compatibility that restrict the possible examples. In \cite{balchinCpqr} Balchin, Bearup, Pech and Roitzheim computed the size of the poset $I_{C_{pqr}}$ of $C_{pqr}$-transfer systems for $p, q, r$ three different primes. This poset is much bigger than the product $I_{C_p} \times I_{C_q} \times I_{C_r}$.
\end{rem}

\section{Global transfer systems for \texorpdfstring{$\AbLie$}{AbLie}}
\label{sectionAbLie}

\begin{prop}
\label{propablie}
The functors $U_\Ab^\AbLie$ and $R_\Ab^\AbLie$ are inverse isomorphisms between the poset $I_\Ab$ and the poset $I_\AbLie$ of $\AbLie$-transfer systems.
\end{prop}
\begin{proof}
By Proposition~\ref{propadjoint} the functor $U^\AbLie_\Ab \colon I_\AbLie \to I_\Ab$ is left adjoint to $R^\AbLie_\Ab \colon I_\Ab \to I_\AbLie$, and $U^\AbLie_\Ab \circ R^\AbLie_\Ab = \id_{I_\Ab}$. Let $T$ be an $\AbLie$-transfer system, then by adjointness $T \subset R^\AbLie_\Ab(U^\AbLie_\Ab(T))$. Let $S= R^\AbLie_\Ab(U^\AbLie_\Ab(T))$. We only have left to show that $S \subset T$.

Consider an abelian compact Lie group $H$ and a finite index closed subgroup $K$ such that $K \leqslant_S H$, and let $H_0$ denote the identity component of $H$. Since $K$ has finite index in $H$, $H_0 \subset K$. Because $H$ is an abelian compact Lie group, it is isomorphic to $H_0 \times A$, where $A$ is a finite abelian group. Then there exists a section $s \colon H/H_0 \cong A \to H_0 \times A \cong H$ of the projection $\pi \colon H \to H/H_0$. By Lemma~\ref{lemmsection}~i), since $K \leqslant_S H$, also $K/H_0 \leqslant_S H/H_0$. But now $H/H_0 \cong A$ is a finite abelian group and $U^\AbLie_\Ab \circ R^\AbLie_\Ab = \id_{I_\Ab}$, so $K/H_0 \leqslant_T H/H_0$, and using Lemma~\ref{lemmsection}~i) again we obtain that $K \leqslant_T H$, proving that $T  = S = R^\AbLie_\Ab(U^\AbLie_\Ab(T))$.
\end{proof}

The results that we have presented until now can be summarized in the following diagram, where the arrows represent the respective forgetful functors, or products of them. It is straightforward to check that this diagram commutes, since both possible compositions are just given by the forgetful functors to the posets $I_\Cycp$.
\[\begin{tikzcd}
	& {I_\AbLie} \\
	& {I_\Ab} \\
	{\prod_{p \in P} I_\Abp} && {I_\Cyc} \\
	& {\prod_{p \in P} I_\Cycp}
	\arrow["\cong", from=1-2, to=2-2]
	\arrow["\cong", from=2-2, to=3-1]
	\arrow["\cong", from=3-1, to=4-2]
	\arrow[from=3-3, to=4-2]
	\arrow["U^\Ab_\Cyc", from=2-2, to=3-3]
\end{tikzcd}\]

In Propositions~\ref{propabp}, \ref{propab}, and \ref{propablie} we proved that the marked arrows are isomorphisms of posets. We also know that $U^\Ab_\Cyc \circ R^\Ab_\Cyc = \id_{I_\Cyc}$. We can use these two facts and the commutativity of this diagram to obtain that $R^\Ab_\Cyc \circ U^\Ab_\Cyc$ is an isomorphism, and thus $U^\Ab_\Cyc$ is also an isomorphism of posets. Proposition~\ref{propcycp} characterizes the poset $I_\Cycp$ for prime $p$, and by putting together these results we obtain our main theorem.

\begin{thm}
\label{thmmain}
The poset $I_\AbLie$ of global transfer systems of abelian compact Lie groups is isomorphic to the poset $I_\Cyc$ of global transfer systems of finite cyclic groups, which is itself isomorphic to the product $(\prod_{p \in P} \NN \cup \{\infty\})^\opp$, where $P$ is the set of prime numbers.
\end{thm}
\begin{proof}
As we just explained, Propositions~\ref{propabp}, \ref{propab}, and \ref{propablie} together prove that $I_\AbLie \cong I_\Cyc \cong \prod_{p \in P} I_\Cycp$, and by Proposition~\ref{propcycp} this is isomorphic to $(\prod_{p \in P} \NN \cup \{\infty\})^\opp$.
\end{proof}

\begin{rem}
We can use the previous results not only to obtain an explicit description of the poset $I_\AbLie$ of all global transfer systems of abelian compact Lie groups, but also to describe the subgroup pairs contained in each such transfer system. 

The divisibility lattice of $\NN$ embeds into $\prod_{p \in P} \NN \cup \{\infty\}$ as prime factorizations. Thus we can interpret elements of $\prod_{p \in P} \NN \cup \{\infty\}$ as "generalized prime factorizations", which include possibly infinite powers of an infinite amount of different prime numbers. Let $N = \prod_{p \in P} p^{N_p}$ be such an element of $\prod_{p \in P} \NN \cup \{\infty\}$. We will use $T^N$ to denote the global transfer system of abelian compact Lie groups which is the image of $N$ under the chain of isomorphisms that we have constructed. Proposition~\ref{propRTn} characterizes which subgroup pairs belong to a global transfer system of finite abelian $p$-groups. Tracing this back through the chain of isomorphisms we obtain that $K \leqslant_{T^N} H$ for abelian compact Lie groups $K \leqslant H$ precisely if $K$ contains all elements of $H$ whose order divides $N$. We say that a non-zero natural number divides $N$ if its prime factorization is less than or equal to $N$ in $\prod_{p \in P} \NN \cup \{\infty\}$. This completely characterizes the subgroup pairs that belong to $T^N$. Note that if $N \leqslant N'$ as prime factorizations, then $T^{N'} \subset T^N$.
\end{rem}

\section{Non-abelian case}
\label{sectionnonabelian}

One might wonder if the transfers of abelian groups capture all of the information contained in a global transfer system. In this section we will prove that this is not the case, that there are different global transfer systems of compact Lie groups which coincide when restricted to abelian compact Lie groups. To do this we use the description of the global transfer systems for the family of subgroups of $\Sigma_3$ the symmetric group on three elements that we gave in \cite[Example~7.4]{globalNin}.

\begin{prop}
The forgetful functor $U^\Lie_\AbLie$ is not injective. Since $U^\Lie_\AbLie \circ R^\Lie_\AbLie = \id_\AbLie$, this implies that $R^\Lie_\AbLie$ is not surjective.
\end{prop}
\begin{proof}
Let $\fami_{\Sigma_3}$ be the family of groups isomorphic to $\Sigma_3$ or to a subgroup of $\Sigma_3$. Recall that an $\fami_{\Sigma_3}$-transfer system is different to a $\Sigma_3$-transfer system, see Remark~\ref{remdefGtransfer}. Let $\fami_{C_2, C_3}$ be the family of groups isomorphic to either $C_2$, $C_3$, or the trivial group. We have the following diagram of functors.

\[\begin{tikzcd}[row sep=large,column sep=large,every label/.append
style={font=\large},every matrix/.append style={nodes={font=\large}}]
	I_{\fami_{\Sigma_3}} & I_\Lie \\
	I_{\fami_{C_2, C_3}} & I_\AbLie
	\arrow["{U^{\fami_{\Sigma_3}}_{\fami_{C_2, C_3}}}"', from=1-1, to=2-1]
	\arrow["{R^\AbLie_{\fami_{C_2, C_3}}}"', from=2-1, to=2-2]
	\arrow["{R^\AbLie_{\fami_{\Sigma_3}}}"{description}, from=1-1, to=2-2]
	\arrow["{U^\Lie_\AbLie}", from=1-2, to=2-2]
	\arrow["{R^\Lie_{\fami_{\Sigma_3}}}", from=1-1, to=1-2]
\end{tikzcd}\]

The upper triangle in the diagram commutes by Lemma~\ref{lemmRcomp}, since $U^\Lie_\AbLie=R^\AbLie_\Lie$ and any homomorphism from a subgroup of $\Sigma_3$ to an abelian compact Lie group obviously factors through a compact Lie group. The lower triangle similarly commutes by Lemma~\ref{lemmRcomp}, since any homomorphism from a subgroup of $\Sigma_3$ to an abelian compact Lie group factors through $C_2$ or $C_3$.

In \cite[Example 7.4]{globalNin} we computed $I_{\fami_{\Sigma_3}}$ the poset of $\fami_{\Sigma_3}$-transfer systems. There are five $\fami_{\Sigma_3}$-transfer systems, and two of them coincide when restricted to the family $\fami_{C_2, C_3}$. The first one is $\All$ the greatest $\fami_{\Sigma_3}$-transfer system, such that $K \leqslant_\All H$ for all $H$ isomorphic to a subgroup of $\Sigma_3$ and all $K \leqslant H$. The second one is the $\fami_{\Sigma_3}$-transfer system $T$ such that $K \leqslant_T H$ unless $H$ is isomorphic to $\Sigma_3$ and $K$ has order two. Both of these restrict to the greatest $\fami_{C_2, C_3}$-transfer system, as $\Sigma_3 \notin \fami_{C_2, C_3}$. This means that
\[U^\Lie_\AbLie(R^\Lie_{\fami_{\Sigma_3}}(T))= R^\AbLie_{\fami_{C_2, C_3}}(U^{\fami_{\Sigma_3}}_{\fami_{C_2, C_3}}(T)) = R^\AbLie_{\fami_{C_2, C_3}}(U^{\fami_{\Sigma_3}}_{\fami_{C_2, C_3}}(\All)) = U^\Lie_\AbLie(R^\Lie_{\fami_{\Sigma_3}}(\All)).\]
Since $R^\Lie_{\fami_{\Sigma_3}}$ is an embedding, $R^\Lie_{\fami_{\Sigma_3}}(\All)$ and $R^\Lie_{\fami_{\Sigma_3}}(T)$ are different $\Lie$-transfer systems, which coincide when restricted to $\AbLie$.
\end{proof}

With this we have proven that the poset $I_\Lie$ is strictly bigger than the poset $I_\AbLie$. 

\printbibliography

\end{document}